\documentclass[10pt, a4paper]{amsart}

\usepackage{graphicx}

\usepackage{multirow}
\usepackage{array}
\usepackage[caption=false,font=footnotesize]{subfig}
\usepackage{fixltx2e}
\usepackage{url}
\usepackage{amssymb}
\usepackage{amsthm}
\usepackage{hyperref}
\usepackage{verbatim}
\usepackage{enumerate}

\theoremstyle{definition}
\newtheorem{definition}{Definition}
\theoremstyle{remark}

\theoremstyle{definition}
\newtheorem{theorem}{Theorem}

\begin{document}

\title{A Case Study on Regularity in Cellular Network Deployment}

\author{J.S. Gomez,
        A. Vasseur,
        A. Vergne,
        P. Martins,
        L. Decreusefond,
        and~Wei~Chen} 
\thanks{J.S. Gomez is with the Institut Mines-T\'el\'ecom, T\'el\'ecom ParisTech, CNRS, LTCI (Paris, France) and the Departement of Electronic Engineering of Tsinghua University (Beijing, China), e-mail: jean-sebastien.gomez@telecom-paristech.fr}%
\thanks{A. Vasseur, A. Vergne, P. Martins and L. Decreusefond are with the Institut Mines-T\'el\'ecom, T\'el\'ecom ParisTech, CNRS, LTCI (Paris, France), e-mails: \{aurelien.vasseur, anais.vergne, philippe.martins, laurent.decreusefond\}@telecom-paristech.fr }%
\thanks{Wei Chen is with the Department of Electronic Engineering of Tsinghua University (Beijing, China), e-mail: wchen@tsinghua.edu.cn}

\maketitle

\begin{abstract}
  This paper aims to validate the $\beta$-Ginibre point process as a
  model for the distribution of base station locations in a cellular
  network. The $\beta$-Ginibre is a repulsive point process in which
  repulsion is controlled by the $\beta$ parameter. When $\beta$ tends
  to zero, the point process converges in law towards a Poisson point
  process. If $\beta$ equals to one it becomes a Ginibre point
  process. Simulations on real data collected in Paris (France) show
  that base station locations can be fitted with a $\beta$-Ginibre point
  process. Moreover we prove that their superposition tends to a
  Poisson point process as it can be seen from real data. Qualitative
  interpretations on deployment strategies are derived from the model
  fitting of the raw data.
\end{abstract}




\section{Introduction}

{S}{tatistical} models of transmitters locations aim to
provide tools to understand real network deployment. For
telecommunication companies, the \textit{a priori} knowledge of the
distribution of the antenna locations helps to predict and manage the
costs of a network deployment. Such models also provide mathematical
tractable methods to estimate the coverage probability of a given
network. These results would also interest telecommunication
regulators and public health authorities, since electromagnetic
exposure has become a worldwide issue.

The first model introduced in radio networks was the regular hexagonal
deterministic network. Although the regular lattice of cells gives an
approximation of the cellular concept, it fails to catch the proper
reality of network deployment. It proves also to be an optimistic
bound in terms of interference estimation \cite{Baccelli_article}. The
random nature of the parameters involved in defining a proper coverage
strategy makes it difficult to use a deterministic and regular
model. Stochastic geometry ideas, especially about random point
processes -i.e. Poisson Point Processes (PPP), Mat\'ern hard-core
point processes, Ginibre Point Process (GPP) and $\beta$-Ginibre Point
Process ($\beta$-GPP)- were then widely explored in the wireless
communication literature.  Pioneer works in this field were realized
by Baccelli et al. on PPP \cite{Baccelli_livre}. Many results were
then derived, such as the coverage probability in respect of the
SINR. Last developments of PPP models also include modelling of
heterogeneous ($k$-tier) networks \cite{k-tier}. However, positions of
the base stations in a PPP deployed network are uncorrelated with one
another. Therefore clusters of points may occur. Mean inter-site
distance of such configurations is thus smaller than what happens in
reality. As a result, PPP models generate more interference than that
of a real network. The articles of Andrews et al. \cite{Baccelli_article} and Nakata et al. \cite{Nakata20147} show that the PPP provide with the most pessimistic prediction of outage probability compared with other repulsive models. 

Spatial correlations between base stations locations exist, since they
have to be separated from one another to maximize coverage and
minimize inter-site interferences. To take into account these effects,
repulsive (or regular) models were introduced in the literature. A
simple approach is to transform a PPP into a repulsive point process
by thinning. Such processes are called Mat\'ern hard-core point
processes. Interferences for such deployed networks were investigated
in \cite{matern_hard_core} but hard-core models proved to be difficult
to manipulate since the outage probability can not be analytically
deduced.  Soft-core processes then rose community's interest. Among
them, GPP and $\beta$-GPP (two determinantal point processes) were
investigated in the wireless communication field. They were at first
introduced by Shirai et al. \cite{shirai_taka_derterminant} in quantum
physics to model fermion interactions. Works of Miyoshi et
al.\cite{miyoshi2014cellular} and Deng et
al. \cite{haenggi2014ginibre} have derived coverage probability in
respect of the SINR for both GPP and $\beta$-GPP models.

In this paper, we show that base station distribution for an operator
and for a technology can be fitted with a $\beta$-GPP distribution in
the Paris area. The distribution of all base stations of all operators can be
fitted with a PPP. Our main contribution lies in the theoretical
justification of this phenomenon. We prove that the superposition of
different $\beta$-GPPs converges in distribution to a PPP process.
Finally we draw conclusions on the coverage-capacity trade-off made by
different operators. Qualitative results are derived from the inferred
values of $\beta$ and the intensity $\lambda$. $\lambda$ can give information on the dimensioning
strategy adopted by the operator, while $\beta$ give insights on the
coverage.

Other existing papers on antenna deployment models mainly consider the
computation of the SINR and coverage probability for a wide set of
point processes. We are instead interested in validating the
$\beta$-GPP model and the PPP superposition model with real data on a
dense urban area. Such a case study is made possible because French
frequency regulator (ANFR) provides location in an open access
database \cite{Cartoradio}.

In Section II, we define mathematically the $\beta$-GPP and introduce
the convergence in distribution theorem for a superposition of
$\beta$-GPP. In Section III, we give the method used to fit the
$ \beta$-GPP model with the actual data.  A qualitative
interpretation of the deployment strategies is then realized from
inferred $\beta$ and $\lambda$ .

\section{Theoretical Model}

In this section, we recall the definition of the $\beta$-GPP. We also introduce the theorem $\beta$-GPP convergence theorem.   
\subsection{Definitions}
Let $\mathbb{C}$ denote the complex plane. Let $\Phi$ be a realisation of a certain point process.  Let $x_1, \ldots, x_k$, be $k$-tuples of distinct pairwise elements of $\mathbb{C}$.  Let $R \subseteq \mathbb{C}$ be a Borel set and $f: R^n \rightarrow \mathbb{R}_+^*$ be any Borel function.
\begin{definition}[Correlation function] 
The $k$-th joint density function of a point process is defined by: 
\begin{align*}
\mathbb{E} \! \sum_{ \substack{\left( x_1, \ldots x_n \right) \\  x_i  \neq x_j}} & \!  f\left(x_1, \ldots, x_k\right)  \!  \\ = \! \int_{R^n} \! f\left(x_1, \ldots, x_k \right)& \! \rho^{(k)} \! \left(x_1, \ldots, x_k \right) \! dx_1  \! \ldots \! dx_k.
 \end{align*}
\end{definition}
Let $\beta$ be a real number in $]0,1]$, let $\lambda$, a strictly positive real number be the intensity of a point process and $c = \lambda \pi$.
The $\beta$-GPP is a determinantal point process that can be defined by its correlation functions. 
\begin{definition}[$\beta$-Ginibre]
\begin{equation*}
 \rho^{\left(k \right)} \left( x_1, \ldots ,x_k \right) = \det \left( K_{c, \beta} \left( x_i,x_j\right), 1 \leq i,j \leq k \right),
\end{equation*}
where $K_{c, \beta}$ is a kernel such as $\forall (x,y) \in \mathbb{C}^2$:
\begin{equation*}
K_{c, \beta} \left( x,y \right) = \frac{c}{\pi} e^{- \frac{c}{2 \beta} \left( |x|^2 + |y|^2 \right)} e^{ \frac{c}{\beta} x \bar{y} },
\end{equation*}
with the respect of the Lebesgue measure.
\end{definition}
A $\beta$-GPP can also be obtained by thinning and rescaling a GPP. Each point of the GPP is kept independently with probability $\beta$, then a rescaling of ratio $\sqrt{\beta}$ is applied in order to maintain the original intensity. 
If $\beta = 1$, the $1$-GPP is equal to the GPP. For $\beta$ tending to zero, the thinning increases the randomness of the process, and then the kernel $K_{c,\beta}$ tends to a diagonal matrix. Correlation between points disappears. The Laplace transform of a $\beta$-GPP is given for all $f$ by $L(f)= \det(I-\sqrt{1-e^{-f}}K_{c, \beta} \sqrt{1-e^{-f}})$, where the determinant is the one of Carleman-Fredholm. It is easy to see that the $\beta$-GPP Laplace transform converges then to the one of a PPP when $\beta$ goes to zero.

\subsection{Properties of superposition of $\beta$-GPP}
One of the main novelties in this paper is the study of the superposition of multiple $\beta$-GPP. We give the key convergence theorem for the $\beta$-GPP. \\
Let $\Phi_1, \Phi_2,...,$ be point processes in a metric space $R$. Let $\chi$ be the space of all locally finite subsets (configurations) in $R$.
We now introduce the $\beta$-GPP superposition convergence theorem. \\
For all $n \in \mathbb{N}^*$, let $\Phi_n  $ be the superposition of $\left( \Phi_{n,i} \right)_{i \in [|1,n|]}$ family of $\beta_{n,i}$-GPP with intensity $\lambda_{n,i}\!=\!\left(n \pi\right)^{-1} c_i$ and $\beta_{n,i} \in ]0,1]$.
\begin{theorem}[$\beta$-GPP superposition convergence theorem] 
\label{Theo:1}
Let us suppose that:
\begin{enumerate}[(i)]
\item the sequence $(c_i)_{i \in \mathbb{N}^*} \subset \mathbb{R}_+^*$ is bounded,
\item $\underset{n\to +\infty}{\lim} n^{-1} \sum_{i=1}^n c_i$ is finite and equal to $c$.
\end{enumerate} 
Then $\left( \Phi_n \right)_{n \in \mathbb{N^*}} $ converges in distribution to a PPP $\Phi$ with intensity $c \pi^{-1}$.
\end{theorem}
\begin{proof}
This theorem is proven in Appendix \ref{a1}
\end{proof}

Hypotheses (i) and (ii) of Th.\ref{Theo:1} are quite restrictive because the intensities of each $\beta$-GPP are dependent of $n$. However, in practice, we mainly work with finite families of $\beta$-GPP. Therefore, we can choose the value of the $(c_i)_{i \in {1...n}}$ such that they match the real values of the intensity of each $\beta$-GPP. 

\section{Simulations}

In this section we introduce the fitting method that is used to obtain the parameter $\beta$. We also present the results from the fitting of each deployment and operator in Paris, France. 

\subsection{Summary statistic}

In order to fit the real deployment to the $\beta$-GPP model, we introduce the $J$ function that characterize any point process. This function is a summary statistic based on inter-point distances. General information about summary statistics can be found in \cite{moller2003statistical}.  
Let $u$ be any location in the plane $\mathbb{C}$, and $\Phi= \left\lbrace x_i \right\rbrace_{i \in \mathbb{N}}$ be a realization of a $\beta$-GPP.

\begin{definition}[$J$-function of the $\beta$-GPP]
\begin{equation*}
    \begin{split}
		\forall r \in \mathbb{R}_+^* \; \;  J(r) &:= \frac{1-G(r)}{1-F(r)}, \\
			&= (1 - \beta + \beta e^{-\frac{c}{\beta} r^2 })^{-1},
	\end{split}
\end{equation*} 
where $G$ and $F$ are respectively the contact distribution function and the nearest-neighbor distance distribution function. 
\end{definition}

The $J$-function characterizes the repulsiveness or attractiveness of a point process. If the $J$-function is bigger than one, the point process is repulsive, otherwise, it is attractive. By definition of a PPP, $J(r) \equiv 1$ In the case of the $\beta$-GPP, the $J$-function is always bigger than one on its definition domain. This confirms that the $\beta$-GPP is a repulsive point process. It characterizes the $\beta$-GPP and its expression is proved in \cite{haenggi2014ginibre}.
\subsection{Fitting method}
\begin{figure}[h]
\centering
\includegraphics[width=88mm]{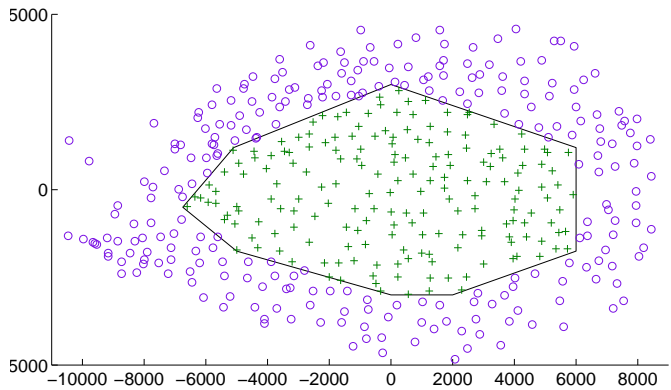}
\caption{Example of data sample for one GSM operator. The $J$-function is fitted on the points within the rectangular window.}
\label{Figure:1}
\end{figure}

Thanks to R language and the spatstat package \cite{baddeley2005spatstat}, the estimate of the $J$-function is derived from the raw data. Since we consider only a finite set of antennas, edge-effect might appear on the $J$-function estimate. We then have to keep a subset of the data to perform the estimation. Figure \ref{Figure:1} gives the window we considered for extracting data in Paris, France. It covers about 60\% of the city and its shape is chosen to match the geographical borders. The values of the $J$-function estimate are computed for $r \leq 600 \; m$. Above $600 \; m$, the estimation is not relevant due to the edge-effect. $J$ is then directly fitted on the estimate and the parameter $\beta$ is deduced. An example of fitting is given in Figure \ref{Figure:2}. It is clear that the point process formed by the base stations locations is repulsive and fits well the theoretical model. Therefore, it outfits the PPP model, because the $J$-function a PPP is equal to one for all $r$. In the next paragraph we present the results we obtained on raw data.

\subsection{Fitting results and interpretation}
\begin{figure}[h]
\centering
\includegraphics[width=88mm]{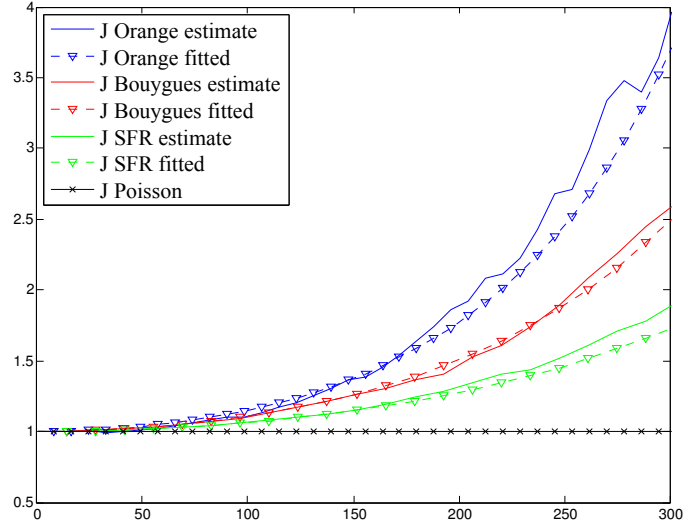}
\caption{Example of $J$-function fitting for Orange, SFR and Bouygues on the 3G 2100 MHz band.  As a comparison, $J(r) = 1$ for all $r$ in the PPP case.}
\label{Figure:2}
\end{figure}

\begin{table}[h]
\renewcommand{\arraystretch}{1.3}
\caption{Numerical values of $\beta$ and $\lambda$ per technology and operator}
\label{table:1}
\centering
\begin{tabular}{cc |c c c c|}

\cline{3-6}                                                  &           & Orange & SFR & Bouygues & Free \\ \hline 
\multicolumn{1}{|c}{\multirow{2}{*}{GSM 900}}    & $\beta$   & 0.81 & 0.76 & 0.65 & NA \\  
\multicolumn{1}{|c}{}                            & $\lambda$ & 2.39 & 2.65 & 2.63 & NA \\ \hline
\multicolumn{1}{|c}{\multirow{2}{*}{GSM 1800}}   & $\beta$   & 0.84 & 0.85 & 0.71 & NA \\ 
\multicolumn{1}{|c}{}                            & $\lambda$ & 3.00 & 2.39 & 3.59 & NA \\ \hline
\multicolumn{1}{|c}{\multirow{2}{*}{UMTS 900}}   & $\beta$   &  NA  & 0.97 & 0.53 & 0.89 \\ 
\multicolumn{1}{|c}{}                            & $\lambda$ &  NA  & 1.92 & 2.44 & 1.05 \\ \hline
\multicolumn{1}{|c}{\multirow{2}{*}{UMTS 2100}}  & $\beta$   & 1.04 & 0.65 & 0.82 & 0.89 \\ 
\multicolumn{1}{|c}{}                            & $\lambda$ & 3.27 & 3.48 & 4.04 & 1.05 \\ \hline
\multicolumn{1}{|c}{\multirow{2}{*}{LTE 800}}    & $\beta$   & 1.02 & 0.93 & 0.67 & NA \\ 
\multicolumn{1}{|c}{}                            & $\lambda$ & 0.67 & 1.65 & 1.87 & NA \\ \hline
\multicolumn{1}{|c}{\multirow{2}{*}{LTE 1800}}   & $\beta$   &  NA  &  NA  & 0.75 &  NA \\ 
\multicolumn{1}{|c}{}                            & $\lambda$ &  NA  &  NA  & 3.46 &  NA \\ \hline
\multicolumn{1}{|c}{\multirow{2}{*}{LTE 2600}}   & $\beta$   & 0.93 & 0.67 & 0.63 & 0.89 \\ 
\multicolumn{1}{|c}{}                            & $\lambda$ & 2.80 & 2.76 & 2.46 & 1.05 \\ \hline
\end{tabular}
\end{table}

Locations of the base stations are publicly available for the whole French territory and can be found online \cite{Cartoradio}. There are four operators in France and most of them provide 2G to 4G coverage. For each operator and each technology, numerical values of $\beta$ and $\lambda$ from the fitting are given in Table \ref{table:1}. 

Values of $\beta$ and $\lambda$ give some insights about the deployment strategy of each cellular network operators, especially about the coverage-capacity trade-off. Orange's high values of $\beta$ and $\lambda$ suggest that this operator deployed (as the historic, previously state-owned operator) a network that fulfilled an optimal coverage and an optimal traffic capacity (densely deployed network). However, SFR and Bouygues first deployed a network with a minimum of antennas (in order to abide by the coverage requirement of the regulator) and then gradually increased traffic capacity on hot-spots (by increasing locally the number of antennas). This involves adding more antennas on sites that are already covered, thus creating clusters and decreasing the value of $\beta$ and increasing the value of $\lambda$. The French telecommunication regulator (ARCEP) published yearly reports \cite{ARCEP} that suggest such evolution.

We deduce that French operators used two different deployment
strategies. The first strategy consists in fulfilling both coverage
and optimal traffic capacity at once. While the second strategy is to
deploy a network that abides to the coverage requirements in a first
stage, then in a second stage to increase the number of antennas on
hot-spots in order to improve the traffic capacity.

\begin{table}[h]
\renewcommand{\arraystretch}{1.3}
\caption{Numerical values of $\beta$ and $\lambda$ per operator and for the superposition of all the sites}
\label{table:2}
\centering
\begin{tabular}{c|c c c c c|}
\cline{2-6}                                      & Orange & SFR  & Bouygues & Free & \textbf{Superposition} \\ \hline
\multicolumn{1}{|c|}{$\beta$}            & 0.94   & 0.70 & 0.81     & 0.89 & \textbf{0.17}  \\
\multicolumn{1}{|c|}{$\lambda$}       & 3.48   & 3.70 & 4.23     & 1.05 & \textbf{10.28} \\ 
\multicolumn{1}{|c|}{Number of sites} & 185    & 197  & 225      & 56   & \textbf{547}   \\ \hline
\end{tabular}
\end{table}

When deploying their 3G or 4G networks, operators reused and shared
some existing 2G sites. Therefore, we consider that classifying the
base station sites per operator is more relevant than classifying them
by technologies. Table \ref{table:2} summaries these results. As
expected, previous conclusions still hold as values of $\beta$ are
stable between the two tables. We also notice that Free, as a newcomer
(2012), has a small amount of traffic to deal with, and therefore has
deployed less antennas than its competitors. Data analysis also shows
that the superposition of all sites is tending to a PPP as $\beta$ is
equal to $0.17$. Therefore the PPP model still holds as an indicator
of electromagnetic exposure of cellular networks.

\section{Conclusion}

In this paper, we successfully show that $\beta$-GPP is a realistic
model for base station distribution. The $\beta$ parameter is inferred
by using statistical tools on real data. Qualitative results on
network deployment are then derived. We also prove theoretically that
the superposition of multiple $\beta$-GPPs converges in distribution
to a PPP justifying observations made on real deployments. This will
have greater implications in modelling multi-tiers networks. We show
that the values of $\lambda$ and $\beta$ are characteristics of the
coverage-capacity trade-off. Future works will investigate the impact
$\lambda$ and $\beta$ on the design of optimal deployment strategies.

\appendix

\section{Proof of Theorem 1}

Let $A$ be a compact subset in $\mathbb{C}$. For a realization of a point process $\Phi$, the random variable $\Phi(A)$ is the number of points in the compact $A$. 
\begin{theorem}
\label{Theo:2}[Convergence in distribution theorem]
For any $A$ compact subset in $\mathbb{C}$, if the three following properties hold:
\begin{enumerate}[(i)]
\item $\underset{ n \to + \infty}{\lim} \mathbb{P} (\Phi_n(A)=0)=\mathbb{P} (\Phi(A)=0)$
\item $\underset{n\to +\infty}{\limsup}\ \mathbb{P}(\Phi_n(A)\leq1)\geq\mathbb{P}(\Phi(A)\leq1)$
\item $\underset{t\to +\infty}{\lim}\ \underset{n\to +\infty}{\limsup}\ \mathbb{P}(\Phi_n(A)>t)=0$
\end{enumerate}
Then: $\Phi_n \xrightarrow{d} \Phi$.
\end{theorem}
\label{a1}
Th.\ref{Theo:1} is achieved if all conditions of Th.\ref{Theo:2} are satisfied. \\
\textit{Condition (iii)}. Thanks to Markov inequality, and since $\left (\mathbb{E}[\Phi_n(A)] \right)_{n \in \mathbb{N}^*}$ is bounded, (iii) holds.
\\
\\
\textit{Conditions (i) and (ii)}. For a Poisson point process, we know that:
\begin{align*}
\mathbb{P} \left( \Phi (A) \! = \! 0 \right) \! &= \! e^{- |A| c \pi^{-1} }, \\
\mathbb{P} \left( \Phi (A) \! \leq \! 1 \right) \! &= \! e^{- |A| c \pi^{-1} } \! \left( 1 \! + \! |A| c \pi^{-1} \right).
\end{align*}
We have yet to calculate the left-hand side of both inequalities (i) and (ii). Let $K_{n,i}$ be the kernel of a $\beta_{n,i}$-GPP. Proposition 3 of Goldman's paper \cite{goldman2010palm} states that:
\begin{align*}
&\mathbb{P}(\Phi_{n,i}(A) \! = \! 0)  \! = \! 1 \! + \! \sum_{p\geq1} \! \frac{(-1)^p}{p} \! \int\limits_{A^p} \! \det[K_{n,i}](v_1,...,v_p) \ell^p(dv), \\
&\mathbb{P}(\Phi_{n,i}(A) \! = \! 1) \! = \! \mathbb{P}(\Phi_{n,i}(A) \! = \! 0) \! \int\limits_A \! R_{n,i}(z)\ell(dz),
\end{align*}
where $\ell$ designates the Lebesgue measure and 

\begin{equation*}
R_{n,i}(z) \! = \! K_{n,i}(z,z) \! + \! \sum_{j\geq2}K_{n,i}^{(j)}(z,z).
\end{equation*}
By hypothesis of Th.\ref{Theo:1}, $\left( c_i \right)_{i \in \mathbb{N}^*}$ is bounded. We also know that $\| K_{n,i}\|_{\infty} = c_i (n \pi)^{-1}$. We can then prove recursively for all $p \geq 1$, there exists a $M > 0$ such that for all $i \in \mathbb{N}^*$,
\begin{equation*}
\label{bound}
0\leq \det[K_{n,i}](v_1,...,v_p) \leq  \| K_{n,i}\|_{\infty}^p \leq \left( \frac{M} {n \pi} \right)^p.
\end{equation*}
Therefore there exists two bounded sequences $(\epsilon_n)_{n \in \mathbb{N}^*}$ and $(\epsilon'_n)_{n \in \mathbb{N}^*}$ independent of $i$ and , such that:
\begin{align*}
&\mathbb{P}(\Phi_{n,i}(A) \! = \! 0) = 1 \! - \! \frac{c_i|A|}{n\pi} \! + \! n^{-2} \epsilon_n.,\\
&\mathbb{P}(\Phi_{n,i}(A) \! = \! 1) = \frac{c_i|A|}{n\pi}+ \! n^{-2} \epsilon'_n.
\end{align*} 
Hence,
\begin{align*}
&\mathbb{P}(\Phi_n(A)\!=\!0) \! = \! e^{O(n^{-2})}e^{-\sum_{i=1}^n\frac{c_i|A|}{n\pi}}, \\
&\mathbb{P}(\Phi_n(A)\!=\!1) \! \geq \! e^{o(\frac{1}{n})} e^{-\sum_{j=1}^n\frac{c_j|A|}{n\pi}} \sum_{i=1}^n \frac{c_i|A|}{n\pi} +o(1).
\end{align*}
Therefore,
\begin{align*}
&\lim_{n \to \infty} \mathbb{P}(\Phi_n(A)\!=\!0) = e^{-|A|c \pi^{-1}}, \\
&\limsup_{n \to \infty} \mathbb{P}(\Phi_n(A)\!=\!1) \geq \frac{c|A|}{\pi} e^{-|A|c \pi^{-1}}, 
\end{align*}
consequently (i) and (ii) hold.
\qed

\end{document}